\documentclass[11pt]{article}

\usepackage{amssymb,amsmath}
 \usepackage{fullpage}
\usepackage{tikz}
\usepackage[left,mathlines,modulo]{lineno}

\newtheorem{theorem}{\bf Theorem}[section]
\newtheorem{corollary}[theorem]{\bf Corollary}

\newtheorem{proposition}[theorem]{\bf Proposition}

\newtheorem{definition}[theorem]{\bf Definition}

\newcommand{\qed}{\hfill $\square$ \bigskip}

\newcommand{\Fib}{{\cal F}}
\newcommand{\Luc}{{\cal L}}

\newcommand{\supp}{{\rm supp}}
\newcommand{\imb}{{\rm imb}}
\newcommand{\irr}{{\rm irr}}

\newcommand{\A}{{\mathcal A}}

\newcommand{\st}{ ~|~ }

\begin{document}
\modulolinenumbers[5]
\title{Some new results about Fibonacci $p$-cubes}

\author{
Michel Mollard\footnote{Institut Fourier, CNRS, UMR 5582, Universit\'e Grenoble Alpes, CS 40700, 38058 Grenoble Cedex 9, France}
}
\date{\today}
\maketitle

\begin{abstract}
The Fibonacci cube $\Gamma_n$ is the subgraph of the hypercube $Q_n$ induced by vertices with no consecutive $1$s. Recently Jianxin Wei and  Yujun Yang introduced a one parameter generalization, Fibonacci $p$-cubes $\Gamma_n^p$, which are subgraphs of  hypercubes induced by strings where there is at least $p$ consecutive $0$s between  two $1$s. In this paper we first prove the expression conjectured by the authors for the cube polynomial of $\Gamma_n^p$. By a totally different method  we then determine a generalization, the distance cube polynomial. We also complete the invariants investigated in the original paper by two new ones, the Mostar index $\mathit{Mo}(\Gamma_n^p)$ and the Irregularity $\irr(\Gamma_n^p)$.
\end{abstract}


\noindent
{\bf AMS Subj. Class. }:  05C12, 05C30, 05C31, 05C60, 05C90

\section{Introduction }

The {\em Fibonacci cube} of dimension $n$, denoted as $\Gamma_n$, is the subgraph of the hypercube $Q_n$  induced by vertices with no consecutive $1$s. 

Fibonacci cubes\ were introduced in 1993 by W.-J.~Hsu\ as a model for interconnection networks~\cite{H-1993a}, but these graphs have found numerous applications elsewhere and are also extremely interesting in their own right. In 2013, a review article on Fibonacci cubes was written by 
Sandi Klav\v{z}ar~\cite{K-2013a} and recently a full book about them has been published~\cite{EKM-2023}.

Structural properties of these graphs have been widely studied.
Fibonacci cubes also play a role in mathematical chemistry. Indeed  they are precisely the resonance graphs of fibonacenes  an important class of hexagonal chains~\cite{KZ-2005a}.
Fibonacci cubes belong to daisy cubes~\cite{KM-2019a} a familly of isometrical subgraphs of hypercubes.
The connections between daisy cubes and  resonance graphs of  catacondensed even ring systems have been explored in~\cite{BTZ-2020}.

Not only the investigation of the 
properties of Fibonacci cubes attracted many researchers, but it
has also led to the development of a variety of interesting 
generalizations and variations covered by a whole chapter in the book
\cite{EKM-2023}. Among these families of graphs are Lucas cubes, 
generalized Fibonacci cubes, Pell graphs, 
$k$-Fibonacci cubes, daisy cubes and Fibonacci-run graphs.

Recently~\cite{WY-2022} Jianxin Wei and  Yujun Yang introduced a one parameter generalization, Fibonacci $p$-cubes $\Gamma_n^p$  which are subgraphs of  hypercubes induced by strings where there is at least $p$ consecutive $0$s between  two $1$s. 
The authors investigated some basic  properties like the numbers of vertices and edges, the diameter, the radius and center. They proposed a conjecture about the expression of the cube polynomial $C_{\Gamma_n^p}(x)$ and various  questions for further investigation.

This paper is organized as follows. After  preliminaries in Section~\ref{sec:basic}, we prove the conjecture about $C_{\Gamma_n^p}(x)$ in Section~\ref{sec:cubepolynomial} and determine a generalization, the distance cube polynomial $D_{\Gamma_n^p}(x,q)$, in Section~\ref{sec:distancecube}. We then consider other invariants, Wienner and Mostar indices in Section~\ref{sec:WM Index}, irregularity $\irr(\Gamma_n^p)$ in Section~\ref{sec:Irr}.

\section{Preliminaries}
\label{sec:basic}
We will next give  some concepts and notations needed in this paper.

We denote by $[a,n]$ the set of integers $i$ such that $a\leq i \leq n$.

Let $(F_n)_{n\geq0}$ be the \emph{Fibonacci numbers}:
$F_0 = 0$, $F_1=1$, $F_{n} = F_{n-1} + F_{n-2}$ for $n\geq 2$.

Let $B=\{0,1\}$. If will be convenient to identify elements $u = (u_1,\ldots, u_n)\in B^n$ and strings of length $n$ over $B$. We thus  briefly write $u$ as $u_1\ldots u_n$ and call $u_i$ the $i$th coordinate of $u$. For $j\in[1,n]$ the string  $u+\delta_j$ will be the string $v$ such that $v_j\neq u_j$ and $v_i=u_i$ for all $i\neq j$.
We will use the power notation for the concatenation of bits, for instance $0^n = 0\ldots 0\in B^n$.

The vertex set of $Q_n$, the \emph{hypercube of dimension $n$},  is the set $B^n$, two vertices being adjacent iff they differ in precisely one coordinate. We will say that an edge $uv$ of $Q_n$ uses the direction $i$ if $u$ and $v$ differ in the coordinate $i$, thus if $v=u+\delta_i$. 

The \emph{distance} between two vertices $u$ and $v$ of  a graph $G$  is the 
number of edges on a shortest shortest $u,v$-path. It is immediate that the distance between two vertices of $Q_n$ is the number of coordinates the strings differ, sometime called Hamming distance.

A {\em Fibonacci string} of length $n$ is a binary string without consecutive 1s. We will call ${\cal F}_n$ the set of Fibonacci strings of length $n$.

The {\em Fibonacci cube} $\Gamma_n$ ($n\geq 1$) is the subgraph of $Q_n$ induced by $\Fib_{n}$ the set of Fibonacci strings of length $n$. Because of the empty string $\lambda$, $\Gamma_0 = K_1$. 

It is well known that for any integer $n$, $|\Fib_{n}|=|V(\Gamma_{n})|=F_{n+2}$.

Many generalizations of Fibonacci cubes have been proposed. Among them daisy cubes we will meet in section~\ref{sec:distancecube} and Fibonacci $(p,r)$-cubes~\cite{EA-1997}. Recently Jianxin Wei and  Yujun Yang studied specifically Fibonacci $p$-cubes which are Fibonacci $(p,r)$-cubes with $r=1$~\cite{WY-2022}. 

For an integer $p\geq1$, A {\em Fibonacci $p$-string} of length $n$ is a binary string where  consecutive 1s are separated by at least $p$ 0s.
 Let ${\cal F}^p_n$ be the set of Fibonacci $p$-strings of length $n$. Then the {\em Fibonacci $p$-cube}, $\Gamma^p_n$ is the subgraph of $Q_n$ induced by $\Fib^p_{n}$. Again because of the empty string  $\Gamma^p_0 = K_1$. Note that the Fibonacci $1$-cube $\Gamma^1_n$ is the classical Fibonacci cubes $\Gamma_n$.

Similarly a binary string $u$ of length $n$ is a {\em Lucas $p$-string} if there are at least $p$ 0s between two 1s of $u$ in a circular manner. The {\em Lucas $p$-cube}, $\Luc^p_{n}$,  is the subgraph of $Q_n$ induced by Lucas $p$-strings.

Let $(F^p_{n})_{n\geq 0}$ be the \emph{Fibonacci $p$-numbers} defined by the recursion
\begin{equation}
F^p_0 = 0, F^p_i=1\text{ for }i\in[1,p]\text{, and }F^p_{n} = F^p_{n-1} + F^p_{n-p-1}\text{ for }n\geq p+1. 
\end{equation}

From this definition, by the usual method, the generating function of the sequence $(F^p_{n})_{n\geq 0}$ is 
$$\sum_{n\geq 0}F^p_{n}x^n=\frac{x}{1-x-x^{p+1}}.$$

It is immediate that $F^p_{p+1}=1$ and more generally $F^p_{n}=n-p$ for $n\in [p+1,2p+2]$. Note that the $(F^1_{n})_{n\geq 0}$ are the Fibonacci numbers.

As noticed in \cite[Theorem~4.1]{WY-2022} the order of $\Gamma^p_n$ is $|{\cal F}^p_n|=F^p_{n+p+1}$.

Indeed  a string in ${\cal F}^p_n$, $n\geq p+1$, can be be uniquely decomposed as the concatenation of $10^p$ with a string of ${\cal F}^p_{n-p-1}$ or as the concatenation of $0$ with a string of ${\cal F}^p_{n-1}$. Furthermore $|{\cal F}^p_n|=n+1=F^p_{n+p+1}$ for $n\leq p$, thus by induction  $|{\cal F}^p_n|=F^p_{n+p+1}$ for any $n$.

The {\em Hamming weight} $w(b)$ of a binary string $b$ is  the number of occurrences of $1$ in $b$. 
It is immediate that the number of strings in ${\cal B}_n$ of Hamming weight $w$ is $\binom{n}{w}$.


Let $u\in {\cal F}^p_n$
with $w(u)= w$. Concatenate $u$ with $0^p$. 
In the string $u0^p$ every $1$ is followed by a substring $0^p$. 
Therefore $u0^p$ is a string over the alphabet $\{0,10^p\}$ and this decomposition is unique.
Replace every substring $10^p$ in $u0^p$ by the character $\ast$. 
By this way we obtain a string $\theta(u)$ 
over the alphabet $\{0,\ast\}$ of length $n+p-wp$ which uses exactly $w$ characters $\ast$. 
It is immediate that $\theta$ is a bijection. Thus we have

\begin{proposition}\label{prop:Hammingweight}
The number of strings in ${\cal F}^p_n$ of Hamming weight $w$ is 
$\binom{n-wp+p}{w}$ and the  maximum weight of a Fibonacci $p$-string of length $n$ is $\left\lfloor \frac{n+p}{p+1}\right\rfloor$.
\end{proposition}

The following result is well-known (see~\cite{HIK-2011} for example).

\begin{proposition}\label{pro:bt}
In every induced subgraph $H$ of $Q_n$ isomorphic to $Q_k$ there exists a unique vertex of minimal Hamming weight, \emph{the bottom vertex} $b(H)$. There exists also a unique vertex of maximal Hamming weight, the \emph{top vertex} $t(H)$. 
Furthermore $b(H)$ and $t(H)$ are at distance $k$ and characterize $H$ among the subgraphs of $Q_n$ isomorphic to $Q_k$. 
\end{proposition}

If $H$ is an induced subgraph of $\Gamma_n$, it is also an induced subgraph of $Q_n$. Thus Proposition~\ref{pro:bt} is still true for induced subgraphs of Fibonacci cubes.
The {\em support} of an induced hypercube $H$, $\supp(H)$, is the set of $k$ coordinates that vary in $H$. Therefore, 
$$\supp(H) = \{i\in [1,n]:\ t_i=1, b_i=0\}\,.$$ 
$H$ is thus characterized by the couple $(t,b)$ consisting of the top vertex and the bottom vertex of $H$.

For a graph $G$, let $c_k(G)$ $(k\ge 0)$ be the number of induced 
subgraphs of $G$ isomorphic to  $Q_k$. The {\em cube polynomial}, $C_G(x)$,
of $G$, is the corresponding enumerator polynomial, that is

\begin{equation}\label{eqn:defCG}
C_G(x) = \sum_{k\geq 0} c_k(G) x^k\,.
\end{equation}
This polynomial was introduced in~\cite{BKS-2003}, determined for Fibonacci and Lucas cubes in~\cite{KM-2012a} and afterwards for several of their variations~(see \cite[Chapter~9]{EKM-2023} for example).

Assume $0^n$ belongs to $G$ subgraph of $Q_n$. A bivariate refinement of $C_{G}(x)$ is the {\em distance cube polynomial} (with respect to $0^n$)~\cite{KM-2019a}, first introduced as {\em$q$-cube polynomial}{~\cite{SE-2017a,SE-2018a} in the case of graphs where it can be seen as a $q$-analogue of the cube polynomial. This polynomial keeps track of the distance of the hypercubes to $0^n$.  We define the polynomial the following way

\begin{equation}\label{eqn:defCGq}
D_G(x,q) = \sum_{k\geq 0} c_{k,d}(G) x^kq^d 
\end{equation} where $c_{k,d}(G)$ is  the number of induced 
subgraphs of $G$ isomorphic to $Q_k$ with bottom vertex at distance $d$ of $0^n$.
%
\section{Cube polynomial of Fibonacci $p$-cubes}
\label{sec:cubepolynomial}

In this section we prove in two different ways the following result conjectured in \cite{WY-2022} as Conjecture 6.3.

\begin{theorem} \label{thm:cubepolyfib}

If $n\ge 0$, then $C_{\Gamma_n^p}(x)$ is of degree $\left\lfloor \frac{n+p}{p+1}\right\rfloor$. Moreover,  
$$C_{\Gamma_n^p}(x)=\sum_{a = 0}^{\left\lfloor \frac{n+p}{p+1}\right\rfloor}\binom{n-ap+p}{a}(1+x)^{a}\,,$$
and the number of induced subgraphs of $\Gamma_n^p$ isomorphic to  $Q_k$ is 
$$c_k(\Gamma_n^p)=\sum_{i = k}^{\left\lfloor \frac{n+p}{p+1}\right\rfloor}\binom{n-ip+p}{i}\binom{i}{k}\,.$$
\end{theorem}
\begin{proof}
The expression for $c_k(\Gamma_n^p)$  can   be deduced by combinatorial arguments as follows. Let $H$ be an induced subgraph of $\Gamma_n^p$ isomorphic to  $Q_k$. The weight of the top vertex of $H$  is at least $k$ and at most $\left\lfloor \frac{n+p}{p+1}\right\rfloor$ by Proposition~\ref{prop:Hammingweight}. Let $i$ such that $k\leq i\leq\left\lfloor \frac{n+p}{p+1}\right\rfloor$. There exist $\binom{n-ip+p}{i}$ Fibonacci $p$-strings $u$ of Hamming weight $i$. 
Choose a set $K$ of $k$ coordinates among the $i$ such that $u_j=1$. A $Q_k$ with top vertex $u$ is induced by the $2^k$ vertices obtained by varying these $k$ coordinates. Therefore $u$ give rise to $\binom{i}{k}$ 
different induced $k$-cubes and the expression for $c_k(\Gamma_n^p)$ follows.
From the  expression of $c_k(\Gamma_n^p)$ exchanging the order of summation in $\sum_{k\geq 0} c_k(\Gamma_n^p) x^k$ it is immediate to deduce $C_{\Gamma_n^p}(x)$.
\end{proof}
\qed

The generating function of $C_{\Gamma_n^p}(x)$ was already determined  by the authors of~\cite{WY-2022}.

\begin{theorem} \label{thm:cubepolyfibgene}\cite[Theorem 6.2]{WY-2022}
The generating function of $C_{\Gamma_n^p}(x)$ is
\begin{equation*}
\sum_{n\geq 0}C_{\Gamma_n^p}(x)t^n=\frac{1+(1+x)t+\ldots+(1+x)t^i+\ldots+(1+x)t^p}{1-t-(1+x)t^{p+1}}\,.
\end{equation*}
%
%
%
\end{theorem}
It is also possible to derive  Theorem~\ref{thm:cubepolyfib} from this generating function. 

Indeed let $$A(y,t)=\frac{1+yt+\ldots+yt^i\ldots+yt^p}{1-t-yt^{p+1}}\,.$$
Expanding $\frac{1}{1-t-yt^{p+1}}$ we deduce that $A(y,t)$ is a bivariate polynomial.   An easy computation give the following equality
\begin{equation}\label{eq:A(y,t)}
t^p A(y,t)=  \frac{1}{1-t-yt^{p+1}}-1-t-\ldots-t^{i}\ldots-t^{p-1}\,.
 \end{equation}
For any integers $a$ and $n$ the coefficient of $y^at^n$  in $ A(y,t)$ is that of $y^at^{n+p}$  in the right side of (\ref{eq:A(y,t)}) thus in $\frac{1}{1-t-yt^{p+1}}$ since $n+p>p-1$.

Note that
\begin{equation*}\frac{1}{1-t-yt^{p+1}}=\sum_{m\geq0 }{(t+yt^{p+1})^m}
\end{equation*}
and $(t+yt^{p+1})^m$ contributes to $y^at^{n+p}$ if and only if  $a(p+1)+(m-a)=n+p$. Furthermore this contribution is $\binom{m}{a}=\binom{n-ap+p}{a}$. Therefore 
\begin{equation}\label{eq:A(y,t)sum}
A(y,t)=\sum_{n\geq0 }{\sum_{a\geq0}}\binom{n-ap+p}{a}{y^at^n}
\end{equation}
with the usual convention that $\binom{b}{a}=0$ for $a>b\,.$

Since $A((1+x),t)$ is the generating function of  $C_{\Gamma_n^p}(x)$ by identification we obtain Theorem~\ref{thm:cubepolyfib}.

We will see in Section~\ref{sec:distancecube} a natural interpretation of $A(y,t)$ and  equality (\ref{eq:A(y,t)}).
\section{Hypercubes in $\Gamma_n^p$ and convolution of $(F_n^p)_{n\geq0}$}\label{sec:convol}
We next study, for a fixed $k$, the sequence $ (c_k(\Gamma_n^p))_{n\geq0}$ of the number of induced subgraphs of $\Gamma_n^p$ isomorphic to  $Q_k$ .

The size of $\Gamma_n^p$ was determined in the seminal paper~\cite[Theorem~4.6]{WY-2022}. An  alternative expression can easily be proved by induction using the relation~\cite[Proposition~3.2]{WY-2022} 
\begin{equation*}
\left|E(\Gamma_n^p)\right|=\left|E(\Gamma_{n-1}^p)\right|+\left|E(\Gamma_{n-p-1}^p)\right|+\left|{\cal F}^p_{n-p-1}\right|\ \ \ \ \ \ \ (n\geq p+1)\,.
\end{equation*}

\begin{proposition}
The size of $\Gamma_n^p,n\geq 0$, is
\begin{equation*}
\left|E(\Gamma_n^p)\right|=\sum_{i=1}^n{F_i^p}F_{n-i+1}^p\,.
\end{equation*}
\end{proposition}

This result is the particular case $k=1$ of the following theorem which have a combinatorial interpretation.
\begin{theorem}
\label{th:product-of-p-Fibo}
For $n,k\ge 0$ the number of induced subgraphs of $\Gamma_n^p$ isomorphic to  $Q_k$ is
$$
c_k(\Gamma_n^p)=\sum_{ {i_0,i_1,\ldots, i_{k} \geq 0 } \atop
i_0+i_1+\dots+i_{k}=n-kp+p+1}\!\!\!\!\!\!\!\!\!\!\!\!\!\!\!\!
F_{i_0}^pF_{i_1}^p\cdots F_{i_{k}}^p\,.
$$
\end{theorem}

To prove this theorem we need first to introduce a notation and two propositions.

\begin{definition}
Let ${\cal F}_n^{p\,\,\bullet}$ be the set of Fibonacci $p$-strings defined by 
 \begin{flalign*}
{\cal F}_n^{p\,\,\bullet}& =\{0^n\} \text{ for }  n\in[0,p-1] \text{ and}\\
{\cal F}_n^{p\,\,\bullet}&=\{v0^p\st v\in {\cal F}_{n-p}^{p}\} \text{ for }n\geq p.
\end{flalign*}
\end{definition}
In other words ${\cal F}^{p\,\,\bullet}=\bigcup_{n\geq0}{\cal F}_n^{p\,\,\bullet}$ is the set of Fibonacci $p$-strings that do not end $10^r$ with $r<p$.

The following property is immediate.
\begin{proposition}
Let $u$ be a binary string,  then $10^pu1$ is a Fibonacci $p$-string if and only if $u\in {\cal F}^{p\,\,\bullet}$.
\end{proposition}
\begin{proposition}\label{pr:car}

$|{\cal F}_n^{p\,\,\bullet}|=F_{n+1}^p$.
\end{proposition}
\begin{proof}
Indeed for $n<p$ $|{\cal F}_n^{p\,\,\bullet}|=1=F_{n+1}^p$ and  
	for $n\geq p$ $|{\cal F}_n^{p\,\,\bullet}|=|{\cal F}_{n-p}^{p}|=F_{n+1}^p$.
\end{proof}\qed

We can prove now Theorem~\ref{th:product-of-p-Fibo}.
The result is true for $k=0$ thus assume $k\geq1$.
For any Fibonacci $p$-string $u$, first note that in the string $u0^p$ each 1 is followed by at least $p$0s.
We already mentioned that we can associate, in a bijective way, to any subcube of $\Gamma_n^p$ with dimension $k$ its top vertex $u$ and its support, a subset $K$ of order $k$ of the set of coordinates $i$ such that $u_i=1$.

We can associate to the choice of the set $K$ a decomposition of the string $u0^p$ as $u0^p=s_010^ps_110^p\dots s_{k-1}10^ps_k $ where the $s_i$'s are $k+1$ substrings  possibly $\lambda$,  whose lengths sum to $n-k(p+1)+p$, such that the $1$s correspond to the coordinates in $K$. 
Since a substring of a Fibonacci $p$-string is a Fibonacci $p$-string, for $i\in[1,k-1]$, $10^ps_i1$ is a Fibonacci $p$-string, thus from proposition~\ref{pr:car}, 
 $s_i\in {\cal F}^{p\,\,\bullet}$. It is immediate that $s_0$ and $s_{k}$ are also in this set.

This construction can be reversed: first we choose the $k+1$ strings in ${\cal F}^{p\,\,\bullet}$, this choice define $K$  and then $u0^p$ by concatenation, separating them with $k$ substrings $10^p$. 
Since the number of strings in ${\cal F}^{p\,\,\bullet}$ of length $i$ is $F_{i+1}^p$, the number of ways to choose the  $k+1$ strings with lengths $i_0,i_1,\dots ,i_k$ is $F_{i_0+1}^pF_{i_1+1}^p\cdots F_{i_{k+1}}^p$ and we obtain
 $$
c_k(\Gamma_n^p)=\sum_{ {i_0,i_1,\ldots, i_{k} \geq 0 }\atop
i_0+i_1+\dots+i_{k}=n-k(p+1)+p}\!\!\!\!\!\!\!\!\!\!\!\!\!\!\!\!
F_{i_0+1}^pF_{i_1+1}^p\cdots F_{i_{k+1}}^p\,.
$$
Considering that $F_0^p=0$  increasing the summation variables by 1, 
the theorem follows.
\qed

In the case $k=1$ a consequence of the proof is the following proposition we will use later.
\begin{proposition}
\label{pr:Ei}
For $i\in[1,n]$ the number of edges in $\Gamma_n^p$ that use the direction $i$ is 

\begin{equation*}
\left|E_i\right|={F_i^p}F_{n-i+1}^p\,.
\end{equation*}
\end{proposition}
From Theorem~\ref{th:product-of-p-Fibo} the sequence $(c_k(\Gamma_n^p))_{n \geq 0}$  is the offset by $kp-p-1$
of the $k$-fold convolution of the sequence $(F_{n}^p)_{n \geq 0}$ with itself. As a consequence we obtain:

\begin{corollary}
\label{co:gen-funct-ck}
Let $k\geq 0$ be an integer. Then the generating function of the sequence 
$(c_k(\Gamma_n^p))_{n\geq 0}$ is
$$\sum_{n \ge 0}c_k(\Gamma_n^p) x^n=\frac{x^{kp-p+k}}{(1-x-x^{p+1})^{k+1}}\,.$$
\end{corollary}

\section{Distance cube polynomial of  Fibonacci $p$-cubes}
\label{sec:distancecube}
This section is devoted to the determination of the generating function of the distance cube polynomial using the weight enumerator polynomial. It also provides, as a specific case, an elegant way to prove Theorem~\ref{thm:cubepolyfibgene}.

Let $\A$ be a set of strings generated freely (as a monoid) by a finite or  infinite alphabet
$E$. 
This means that every strings $s \in \A$ can be written 
uniquely as a concatenation of zero or more strings from $E$. Note that the empty string $\lambda$ belongs to $\A$.
Associate with any string $s \in \A$ a polynomial, possibly multivariate,  $\theta_s$. We will say that $s \rightarrow \theta_s$ is \emph{concatenation multiplicative} over $\A$ if $\theta_{\lambda}=1$ and $\theta_{e_1e_2\ldots e_k}=\theta_{e_1}\theta_{e_2}\ldots\theta_{e_k}$ for any choice of elements $e_i$ in $E$.

Note that if $s \rightarrow \theta_s$ and $s \rightarrow \phi_s$ are concatenation multiplicative then $s \rightarrow \theta_s\phi_s$ is concatenation multiplicative.
Let us recall the following proposition proved in~\cite{M-2024}.

\begin{proposition}\label{pro:mono}

Let $\A$ be a set of strings generated freely, as a monoid, by the alphabet $E$ and $s \rightarrow \theta_s$ concatenation multiplicative over $\A$ then
\begin{equation*}
\sum_{s\in \A}{\theta_s}=\frac{1}{1-\sum_{s\in E}{\theta_s}}\,.
\end{equation*}
\end{proposition}

\begin{definition}
Let $G$ be a subgraph of $Q_n$ and $w_d(G)$, $d\ge 0$, be the number of vertices of $G$ with weight $d$. Then the weight enumerator polynomial of $G$ is the counting polynomial
$$W_{G}(x)=\sum_{u\in V(G)} x^{w(u)}= \sum_{d\geq 0} w_{d}(G) x^d\,.$$
\end{definition}

Let us call {\em extended Fibonacci $p$-strings} the strings of $\bigcup_{n\geq 0}\{u0^p\st u\in{\cal F}^p_n\}$ obtained by adding $0^p$ to Fibonacci $p$-strings.
As noticed in Section~\ref{sec:basic}
 extended Fibonacci $p$-strings can be uniquely obtained as concatenations of strings from $E=\{0,10^p\}$. The monoid generated freely by $E$ is thus
\begin{equation}\label{eq:partition}
\A= \{\lambda\}\cup \{0^i\st 1\leq i \leq p-1\}\cup_{n\geq 0}\{u0^p\st u\in{\cal F}^p_n\}={\cal F}^{p\,\,\bullet}\,.
\end{equation}

Associate with any string binary $s$ the polynomial $\theta_s(x,t)=x^{w(s)}t^{l(s)}$ where $l(s)$ and $w(s)$ are the length and weight of the string, respectively. It is immediate that  $s \rightarrow \theta_s$ is concatenation multiplicative over $\A$ and that $\theta_{s0^p}=t^p\theta_{s}$.  Since   $\sum_{s\in E}{\theta_s}=t+xt^{p+1}$, using  Proposition~\ref{pro:mono}, we obtain 
\begin{equation}\label{eq:somoverA}
\sum_{s\in \A}{\theta_s}=\frac{1}{1-t-xt^{p+1}}\,.
\end{equation}
From the partition~(\ref{eq:partition})  and the equality~(\ref{eq:somoverA}) we obtain
\begin{equation}\label{eq:geneW}
\sum_{n\geq0}\sum_{s\in{\cal F}^p_n}t^p{\theta_s}=\frac{1}{1-t-xt^{p+1}}-1-t\ldots-t^{i}-\ldots-t^{p-1}\,.
\end{equation}
The right side of (\ref{eq:geneW}) can be simplified as
\begin{equation*}
\frac{t^p+xt^{p+1}+\ldots+xt^{p+i}+\ldots+xt^{2p}}{1-t-xt^{p+1}}=t^p\frac{1+xt+\ldots+xt^i+\ldots+xt^p}{{1-t-xt^{p+1}}}\,.
\end {equation*}
The left side is
\begin{equation*}
t^p\sum_{n\geq0}\sum_{s\in{\cal F}^p_n}{\theta_s}=t^p\sum_{n\geq0}\sum_{s\in V(\Gamma^p_n)}{x^{w(s)}t^n}=t^p\sum_{n\geq0}{W_{ \Gamma^p_n}(x)t^n}\,. 
\end {equation*}
Therefore the generating function of $W_{ \Gamma^p_n}(x)$ is given by the following result.
\begin{theorem} \label{thm:weightfibgene}
The generating function of $W_{\Gamma_n^p}(x)$ is
\begin{equation*}
\sum_{n\geq0}{W_{\Gamma_n^p}(x)t^n}=\frac{1+xt+\ldots+xt^i+\ldots+xt^p}{1-t-xt^{p+1}}=A(x,t)\,.
\end{equation*}
\end{theorem}

If $u$ and $v$ are vertices of a graph $G$, the \emph{interval} $I_G(u,v)$ between $u$ and $v$ (in $G$) is the set of vertices lying on shortest $u,v$-path, that is, $I_G(u,v) = \{w | d(u,v) = d(u,w) + d(w,v)\}$. We will also write $I(u,v)$ when $G$ will be clear from the context. 
A subgraph $G$ of a graph $H$ is an \emph{isometric subgraph} 
if the distance between any vertices of $G$ equals the distance 
between the same vertices in $H$. 
Isometric subgraphs of hypercubes are called \emph{partial cubes}. 
The {\em dimension} of a partial cube $G$ is the smallest integer
$d$ such that $G$ is an isometric subgraph of $Q_d$. 
Many important classes of graphs are partial cubes, 
in particular trees, median graphs, benzenoid graphs, phenylenes, 
grid graphs and bipartite torus graphs. In addition, Fibonacci  
and Lucas cubes are partial cubes as well, see \cite{K-2005}.

If $G=(V(G),E(G))$ is a graph and $X\subseteq V(G)$, then $\langle X\rangle$ denotes the subgraph of $G$ induced by $X$.  
Let $\le$ be a partial order on $B^n$ defined with $u_1\ldots u_n \le v_1\ldots v_n$ if $u_i\le v_i$ holds for $i\in [1,n]$. For $X \subseteq B^n$ we define the graph  $Q_n(X)$ as the subgraph of $Q_n$ with  
$$Q_n(X) = \left\langle \{u\in B^n| u\le x\ {\rm for\ some}\ x\in X \} \right\rangle$$
and say that $Q_n(X)$ is the {\em daisy cube generated by $X$}.
 Finally we will say that a graph $G$ is {\em a daisy cube} if there exist an isometrical embedding  of $G$ in some hypercube $Q_n$ and a subset $X$ of $B^n$ such that $G$ is the daisy cube generated by $X$. Such an embedding will be called a {\em proper embedding}.

By construction daisy cubes are partial cubes. It is immediate that the class of daisy cubes is closed under the Cartesian product.

Fibonacci cubes, Lucas cubes, Alternate Lucas-cube~\cite{ESS-2021e} and  Pell graph~\cite{M-2019}\cite[Theorem 9.68 for a proof]{EKM-2023} are examples of daisy cubes.

Since changing a $1$ to $0$ in a Fibonacci $p$-string gives a Fibonacci $p$-string, Fibonacci $p$-cubes are daisy cubes generated by the set of maximal $p$-strings. For the same reason Lucas $p$-cubes belong  also to this class.

if $G$ is a daisy cube, then the polynomials $D_{G,0^n}$ and $C_{G}$ are completely determined by the weight polynomial. More precisely it is proved in~\cite[Corollary 2.6]{KM-2019a} that for a daisy cube
\begin{equation*}
D_G(x,q)=C_G(x+q-1)=W_G(x+q)\,.
\end{equation*}

Therefore replacing $x$ by $x+1$ in the generating function of the weight polynomial obtained in Theorem~\ref{thm:weightfibgene} gives the generating function of the cube polynomial, thus an independent proof of Theorem~\ref{thm:cubepolyfibgene}.

We deduce also that $D_{\Gamma_n^p}(x,q)=D_{\Gamma_n^p}(q,x)$  and replacing $x$ by $x+q$ in the generating function of $ W_{\Gamma_n^p}(x)$ gives that of the distance cube polynomial. 
\begin{theorem} \label{thm:dcubepolyfibgene}
The generating function of $D_{\Gamma_n^p}(x,q)$ is
\begin{equation*}
\sum_{n\geq 0}D_{\Gamma_n^p}(x,q)t^n=\frac{1+(q+x)t+\ldots+(q+x)t^i+\ldots+(q+x)t^p}{1-t-(q+x)t^{p+1}}\,.
\end{equation*}
\end{theorem}
\begin{theorem} \label{th:dcube}
If $n\ge 0$, then,  
\begin{equation*}
D_{\Gamma_n^p}(x,q)=\sum_{a = 0}^{\left\lfloor \frac{n+p}{p+1}\right\rfloor}\binom{n-ap+p}{a}(q+x)^{a}\,,
\end{equation*}
and the number of induced subgraphs of $\Gamma_n^p$ isomorphic to  $Q_k$ at distance $d$ of $0^n$ is 
\begin{equation*}
c_{k,d}(\Gamma_n^p)=\binom{n-(k+d)p+p}{k+d}\binom{k+d}{k}\,.
\end{equation*}
\end{theorem}
\begin{proof}
By adapting the proof of Theorem~\ref{thm:cubepolyfib} the value of $c_{k,d}(\Gamma_n^p)$ can be obtained by combinatorial arguments. Indeed the top vertex $u$ of subgraph isomorphic to a $Q_k$ at distance $d$ of $0^n$ satisfies $w(u)= d+k$. Note that $d+k\leq{\left\lfloor \frac{n+p}{p+1}\right\rfloor}$. The value of $c_{k,d}(\Gamma_n^p)$ follows.

It is then easy to check that $\sum_{k,d} c_{k,d}(\Gamma_n^p)x^kq^d =\sum_{a}\binom{n-ap+p}{a}(q+x)^{a}$ since some $a$ contributes to $x^kq^d$ if and only if only $a=k+d$.
\end{proof}\qed

Like the alternative proof of Theorem~\ref{thm:cubepolyfib} it is also possible to derive $D_{\Gamma_n^p}(x,q)$ and thus $c_{k,d}(\Gamma_n^p)$ from its  generating function setting  $y= x+q$ in the equality~(\ref{eq:A(y,t)sum}).

Considering $q$-analogues of Fibonacci $p$-numbers it is also possible to obtain expression of $c_{k,d}(\Gamma_n^p)$ similar to that of $c_{k}(\Gamma_n^p)$ in Theorem~\ref{th:product-of-p-Fibo} and the generating function of $c_{k,d}(\Gamma_n^p)$ for a fixed $k$.
\section{Wiener index and Mostar Index of Fibonacci $p$-cubes}
\label{sec:WM Index}
The \emph{Wiener index} $W(G)$ of a connected graph $G$ is defined as the sum of all distances between pairs of vertices of $G$. Hence,
$$ W(G) =\sum_{\{u,v\} \subset V(G)}d(u,v).$$
This distance invariant is important in mathematical chemistry. The Wiener index of $\Gamma_n$ has been determined in  ~\cite{KM-2012b}.

Again in the context of graph chemical theory the Mostar index $\mathit{Mo}(G)$ have been introduced in  \cite{DMSTZ-2018} . It measures how far $G$ is from being distance-balanced.
Let $G=(V(G),E(G))$ be a connected graph. For any edge $uv\in E(G)$, let $n_{u,v}$ denote the number of vertices wich are closer to $u$ than to $v$. The \emph{Mostar index} $\mathit{Mo}(G)$  is defined as
$$ \mathit{Mo}(G) =\sum_{uv \in E(G)}|n_{u,v}-n_{v,u}|.$$ 
The Mostar index of Fibonacci cubes has been determined in \cite{ESS-2021c}.

The Wiener index of $\Gamma_n^p$ was already determined~\cite[Theorem~5.7]{WY-2022}
but this result is a particular case of a formula satisfied by daisy cubes. Indeed the Wiener and Mostar indices of daisy cubes are completely determined by $|V(G)|$ and the sequence  $|E_i|$, for $i\in[1,n])$, of the number of edges using the direction $i$.

 \begin{theorem}~\cite[Corollary~4.1]{M-2022a}
\label{th:indices}
Let $G$ be a daisy cube  properly embedded into $Q_n$. For $i\in[1,n]$ let $|E_i|$ be the number of edges using the direction $i$. Then the Wiener and the Mostar indices of $G$ are 
\begin{align*}
 W(G)&=|V(G)||E(G)|-\sum_{i=1}^n |E_i|^2\\
\mathit{Mo}(G)&=|V(G)||E(G)|-2\sum_{i=1}^n |E_i|^2\,.
\end{align*}

\end{theorem}

Therefore using Proposition~\ref{pr:Ei} we obtain one of the expressions for $W(\Gamma_n^p)$ given in~\cite{WY-2022} and the value of $\mathit{Mo}(\Gamma_n^p)$.

 \begin{corollary}
\label{co:indices}
Let $p\geq 1$ and $n\geq 1$. Then the Wiener and the Mostar indices of $G$ are 
\begin{align*}
 W(\Gamma_n^p)&=F_{n+p+1}^p\sum_{i=1}^n{F_i^pF_{n-i+1}^p}-\sum_{i=1}^n{(F_i^pF_{n-i+1}^p)^2}\\
\mathit{Mo}(\Gamma_n^p)&=F_{n+p+1}^p\sum_{i=1}^n{F_i^pF_{n-i+1}^p}-2\sum_{i=1}^n{(F_i^pF_{n-i+1}^p)^2}\,.
\end{align*}
\end{corollary}

 

\section{Irregularity of Fibonacci $p$-cubes}
\label{sec:Irr}
Define the {\em imbalance} $\imb_G(e)$ of an edge $e = uv\in E(G)$ by  
$$\imb_G(e) = |\deg_G(u) - \deg_G(v)|\,.$$

The imbalance of an edge is thus a local measure of non-regularity. To measure graph's global non-regularity, different approaches have been proposed. One of the most natural such measures is the {\em irregularity} $\irr(G)$  introduced by Michael Albertson~\cite{A-1997} as follows: 
$$\irr(G) = \sum_{uv \in E(G)} |\deg_G(u) - \deg_G(v)| = \sum _{e \in E(G)} \imb_G(e)\,.$$

The irregularity of the Fibonacci cubes was first determined by Alizadeh, Deutsch and Klav\v{z}ar.
\begin{theorem} {\rm \cite[Theorem~4.1]{ADK-2020a}}
\label{thm:irregularity-of-Gamma_n} 
If $n\ge 1$, then 
$$\irr(\Gamma_n) =  2 |E(\Gamma_{n-1})|\,.$$
\end{theorem}

A bijective proof of this remarkably simple relation was given by Mollard~\cite{M-2021b}. In this section we will generalize this second approach to Fibonacci $p$-cubes and we prove the following theorem which is the corresponding result for them.

\begin{theorem}
\label{thm:irregularity-of-Gamma_np} 
If $n\ge p$, then 
$$\irr(\Gamma_n^p) =  2 \sum_{d=1}^p|E(\Gamma_{n-d}^p)|\,.$$
\end{theorem}

Let $G$ be an induced subgraph of $Q_n$. Let $e=xy$ with $x_i=1$ and $y_i=0$. An edge $e'=y(y+\delta_j)$ of $G$ will be called an \emph{imbalanced edge for} $e$ if $x+\delta_j \notin V(G)$ (and thus $x(x+\delta_j)\notin E(G)$). Note that such couple of edges does not exist for $G= Q_n$. We will prove that, for $G=\Gamma_n^p$,  the irregularity is the number of imbalanced edges.   
\begin{proposition}\label{pr:zero}
Let $x,y$ be two strings in $\Fib_n^p$ with $y=x+\delta_i$ and $x_i=1$. Then for all $j\in[1,n]$ we have
 $$ x+\delta_j\in \Fib_n^p \text { implies }y+\delta_j\in \Fib_n^p.$$
 \end{proposition}
\begin{proof}
 This is true for $j=i$. Assume $j \neq i $  and $y+\delta_j\notin \Fib_n^p$. Then, since $y\in \Fib_n^p$, $(y+\delta_j)_j=1$  and $y_k=1$ for some $k\in[1,n]$ with  $|k-j|\leq p$ and $k \neq j$. Therefore $y_j=0$ and thus  $x_j=0$.   But for all $ p\in[1,n]$ $x_p=0$ implies  $y_p=0$. Thus $x_k=1$  and, since   $(x+\delta_j)_j=1$ ,  $x+\delta_j\notin \Fib_n^p$.
\end{proof}\qed
\begin{proposition}\label{pr:delta}
Let $x,y$ be two strings in $\Fib_n$ with $y=x+\delta_i$.
Then for all $j\in[1,n]$ with $|i-j|>p$ we have $$x+\delta_j\in \Fib_n^p\text { if and only if }y+\delta_j\in \Fib_n^p.$$
 \end{proposition}
\begin{proof}
This is immediate for $j=i$ thus assume $j \neq i$.
\begin{itemize}
\item If $x_j=1$  then $y_j=x_j=1$ and both $x+\delta_j$ and $y+\delta_j$ belong to $\Fib_n^p$.
\item Assume  $x_j=0$ thus $y_j=0$. We have
$$x+\delta_j\in\Fib_n^p\text{ if and only if } x_{k}=0 \text{ for all } k\in [1,n] \text{ with } |k-j|\leq p$$ and
$$y+\delta_j\in\Fib_n^p\text{ if and only if } y_{k}=0 \text{ for all } k\in [1,n] \text{ with } |k-j|\leq p.$$
But $|i-j|>p$ thus $x_k=y_k$ for all $k$ with $|k-j|\leq p$ and the two conditions are equivalent.
\end{itemize}
\end{proof} \qed

Let $e=xy$ be an edge of $\Fib_n^p$ with $x_i=1$ and $y_i=0$. The degree of $y$ is the number of $j\in [1,n]$ such that $y+\delta_j\in \Fib_n^p$ thus from Proposition~\ref{pr:zero} $$\imb(e) = |\deg(x) - \deg(y)|= \deg(y) - \deg(x) $$  and $\imb(e)$ is the number of imbalanced edges for $e$.

Let $e'=y(y+\delta_j)$ be an imbalanced edge for $e$. Then from Proposition~\ref{pr:delta} $j=i+d$ or $j=i-d$ for some $d$ with $1\leq d\leq p$. If $i<d+1$ or $i>n-d$, then clearly only one case is possible. We say that $e'$ is a 
$d$-{\em right imbalanced edge} for $e$ if $j=i+d$, otherwise $e'$ is a $d$-{\em left imbalanced edge} for $e$. 
For $d\in[1,p]$ let 
$$R_{d} = \{(e,e') \st  e\in E(\Gamma_n^p),\ e'\ {\rm is\ a\ }d{\rm-right\ imbalanced\ edge\ for}\ e\}$$
and
$$L_{d} = \{(e,e') \st  e\in E(\Gamma_n^p),\ e'\ {\rm is\ a\ }d{\rm-left\ imbalanced\ edge\ for}\ e\}\,.$$
We now claim that $|R_{d}| = |E(\Gamma_{n-d}^p)|$. Let 
$$\alpha: R_{d} \rightarrow E(Q_{n-d})$$ 
be defined as follows. If $(e,e')\in R_{d}$, where $e=xy$ with $x_i = 1$, $y_i= 0$, note that $i\leq n-d$ furthermore $x_{i+k}=0$ for $k\in [1,d-1]$. Then set 
$$\alpha((e,e'))= (x_1\ldots x_{i-1}1x_{i+d+1}\dots x_{n})(x_1\dots x_{i-1}0x_{i+d+1}\dots x_{n})\,.$$
Since $e$ is an edge in $\Gamma_n^p$  then $x_{i-k}=0$ for $k\leq p$. Similarly, because $e'$ uses direction $i+d$, $x_{i+d+k}=0$ for $k\leq p$. Therefore $x_1\ldots x_{i-1}1x_{i+d+1}\dots x_{n}$ is a Fibonacci $p$-string and $\alpha((e,e'))\in E(\Gamma_{n-d}^p)$.

Conversely consider an edge $z(z+\delta_i)$ of $\Gamma_{n-d}^p$. and let $x$ and $y$ binary strings of length $n$ defined by 
\begin{align*}
&x_l=y_l=z_l \text{ for } 1 \leq l \leq i-1\\
&x_l=y_l=0 \text{ for } i+1 \leq l \leq i+d-1\\
&x_l=y_l=z_{l-d} \text{ for } i+d+1\leq l \leq n\\
 &x_i=1, y_i=0, x_{i+d}=0, y_{i+d}=1\,.
\end{align*}
It is immediate too verify that $x$, $y$ and $y+\delta_{i+d}$ are Fibonacci $p$-strings but not $x+\delta_{i+d}$,  and thus $(x(x+\delta_i),y(y+\delta_{i+d})) \in R_{d}$.
Therefore  the edge $z(z+\delta_i)$ is the image by $\alpha$ of a couple of edges in $R_{d}$. 
By the properties of imbalanced edges no other choice is possible for $x$ and $y$, thus $\alpha$ is one to one.

Similarly, by a symmetrical argument,  $|L_{d}| = |E(\Gamma_{n-d}^p)|$ and Theorem~\ref{thm:irregularity-of-Gamma_np} follows.
\section{Conclusion}

It is interesting to note that, in many ways,  binary strings can be seen as Fibonacci 0-strings and $\Gamma_{n}^0$ as $Q_n$. Indeed setting p=0 in all theorems of this paper gives results well known for hypercubes like $|E(Q_n)|=n 2^{n-1}$, $C_{Q_n}(x)=(2+x)^n$, $D_{Q_n}(x,q)=(1+x+q)^n$, $W(Q_n)=n2^{2n-2}$, $\mathit{Mo}(Q_n)=0$ or $\irr(Q_n)=0$. The Fibonacci $0$-numbers $(F_n^0)_{n\geq1}$ being $(2^{n-1})_{n\geq1}$.

\bibliographystyle{plain}  

\bibliography{fibo-bib}

\end{document}